\documentclass{amsart}
\usepackage{graphicx} 
\usepackage[utf8]{inputenc}
\usepackage{amsfonts, amsmath, amssymb, amsthm}
\usepackage{dsfont}
\usepackage[margin=1in]{geometry}
\usepackage{comment}
\usepackage{graphicx}
\usepackage{pgffor, pgfplots, mathtools}
\usepackage{enumerate,lineno,setspace,float}
\usepackage{framed}
\usepackage{hyperref}
\usepackage{color}
\usepackage[numbers, sort&compress]{natbib}
\usepackage[capitalize]{cleveref}
\usepackage{doi}
\usepackage{soul}
\usepackage[linesnumbered,ruled]{algorithm2e}
\usepackage{array}
\usepackage{dsfont}
\usepackage{booktabs}

\usepackage{amsthm}
 \usepackage{tikz,ifthen}
 \usetikzlibrary{matrix}
 \usetikzlibrary{arrows.meta}

\newcolumntype{L}[1]{>{\raggedright\arraybackslash}p{#1}}

\hypersetup{
    colorlinks=true,       
    linkcolor={red!50!black},        
    citecolor={green!50!black},        
    urlcolor={blue!50!black}
}

\newcommand{\Href}[1]{\hyperref[#1]{\Cref{#1}}}
\renewcommand{\href}[1]{\hyperref[#1]{\ref{#1}}}
\renewcommand{\eqref}[1]{\hyperref[#1]{(\ref{#1})}}

\newtheorem{theorem}{Theorem}
\newtheorem{lemma}[theorem]{Lemma}

\newtheorem{definition}[theorem]{Definition}

\newtheorem{claim}[theorem]{Claim}

\newtheorem{obs}[theorem]{Observation}
\newtheorem{cor}{Corollary}
\newtheorem*{rem*}{Remark}
\numberwithin{theorem}{section}
\newtheorem*{rep@theorem}{\rep@title}
\newcommand{\newreptheorem}[2]{%
\newenvironment{rep#1}[1]{%
 \def\rep@title{#2 \ref{##1}}%
 \begin{rep@theorem}}%
 {\end{rep@theorem}}}
\makeatother

\newreptheorem{theorem}{Theorem}
\newreptheorem{lemma}{Lemma}
\newreptheorem{claim}{Claim}
\newreptheorem{proposition}{Proposition} 
\newreptheorem{obs}{Observation}

\def\cC{\ensuremath{\mathcal{C}}}

\def\cS{\ensuremath{\mathcal{S}}}

\let\epsilon=\varepsilon
\usepackage{color}
\definecolor{red}{rgb}{1,0,0}
\definecolor{Gold}{rgb}{1,0.843,0.2}
\definecolor{DarkBlue}{rgb}{0,0,0.6}

\pgfplotsset{compat=1.18}
\newcommand{\Xsum}{|X_1|+|X_2|+\ldots+|X_n|}
\newcommand{\wXsum}{|X_1|+2|X_2|+\ldots +n|X_n|}

\newcommand{\tXsum}{|X_2|+|X_3|+\ldots+|X_n|}

\title{An Upper Bound on the Hat Guessing Number of Graphs}

\author{Mason Shurman}
\author{Scott Albert Sibley}

\address{Department of Mathematics, University of California, 
    Irvine, CA, USA}
    
\thanks{The authors would like to thank Asaf Ferber for his comments and mentorship.}

\email{\{mshurman|sibleys\}@uci.edu}
\date{July 2026}

\begin{document}
\maketitle
\begin{abstract}
     The hat guessing number $HG(G)$ of a graph is defined by the following game: each player is placed on a vertex and assigned a hat with one of $k$ colors. Each vertex can see only the hat color of the other vertices it is connected to in $G$. All vertices guess, simultaneously, the color of their own hat. 
     The hat guessing number $HG(G)$ is the largest $k$ such that the players can guarantee that at least one of them guesses correctly. 
     In this paper, we show a general bound on the hat guessing number of a graph $G$ as a function of its order $n$ and its maximum degree $\Delta$. This is the first nontrivial upper bound on $HG(G)$ as a function of $\Delta$ and $n$ when $\Delta \geq \frac{n}{e}$. From this result we also obtain that the hat guessing number of the random graph $G_{n,1/2}$ is at most asymptotically $cn$ for $c\sim 0.809$, and that graphs with maximum degrees of $ (1-\varepsilon )n$ for fixed $\varepsilon>0$ cannot have $HG(G)=(1-o(1))n$.
     
\end{abstract}

\section{Introduction}

Consider the following game: Let $G$ be a graph with $n$ vertices, $v_1,\ldots,v_n$. Each vertex is associated with a player and assigned a color from a set $\{1,\ldots, k\}$ of $k$ possible colors.
Each player may only see the hat colors of their neighbors in $G$, excluding themselves, so if player $i$ is identified with $v_i$, they can see the hat color of player $j$ identified with $v_j$ if and only if $v_iv_j\in E(G)$. All the players agree beforehand on a guessing strategy, and must simultaneously guess the color of their own hat, without communication. The goal of the players is to guarantee that under any coloring, at least one player guesses correctly, in which case the players win. Let the hat guessing number $HG(G)$ be the largest $k$ such that the players have a winning strategy. 
The trivial upper bound for the hat guessing number is $n$, since each player guesses correctly with probability $1/k$, where $k$ is the number of colors, so 
if there were more than $n$ colors, the expected number of correct guesses would be less than 1.

The hat guessing number has been studied in many papers including \cite{alon2020hat,bosek2021hat,HeIdoPrzybocki2022,kokhas2019graphs,latyshev2024hat,butler2009hat,latyshev2022hats}.
One immediate question, asked by Alon, Ben-Eliezer, Shangguan, and Tamo in \cite{alon2020hat}, is whether or not the hat guessing number is bounded above by a function of the maximum degree. 
As they note, by folklore, 
a basic application of the Lovász Local Lemma yields that if $G$ has maximum degree $\Delta$, $HG(G) < e\Delta$. (Note that if a coloring is chosen randomly, the event that a particular vertex guesses right depends on at most $\Delta$ other vertices.) 
However, this argument does not completely resolve their question. In the regime where the maximum degree $\Delta$ is large, this bound is worse than the trivial bound of $n$. Until now, there have not been upper bounds on the hat guessing number as a function of the maximum degree in the regime $(\Delta>n/e)$. Our main result is the first such upper bound.

\begin{theorem}\label{thm:main} For any graph $G$ on $n$ vertices with maximum degree $\Delta$,
       $$ HG(G) \leq 1+\frac{\Delta}{2}+ \frac{\sqrt{(\Delta +2)^2 + 2(n^2-3n-n\Delta)}}{2}.$$
\end{theorem}


One of the most interesting applications of our result is to the random graph $G_{n,1/2}$. 
Note that the maximum degree of $G_{n,1/2}$ is $(1+o(1))n/2\geq n/e$, so the Lovász Local Lemma argument does not yield a nontrivial result. In 2019, Bosek et al.\ in \cite{bosek2021hat} showed that with high probability, that is, with probability tending to 1 as $n$ tends to infinity, $$(2+o(1))\log_2(n)\leq HG(G_{n,1/2})\leq n-(1+o(1))\log_2 n.$$ In 2020, Alon and Chizewer in \cite{alon2022hat} showed that with high probability, $$ HG(G_{n,1/2}) \geq n^{1-o(1)} $$ by finding a subgraph with large hat guessing number. Using our main result, we get a better upper bound on $HG(G_{n,1/2})$. 

\begin{cor}\label{cor:Gn1,2} With high probability,
$$HG(G_{n,1/2})\leq (1+o(1)) \left(\frac{1+\sqrt5}{4}\right)n \sim 0.809n. $$
\end{cor} 
It may be initially unclear whether the inequality in Theorem \ref{thm:main} is always better than the trivial bound of $n$. With some algebraic manipulation, we see that not only is it always better than the trivial bound for non-complete graphs, but if $\Delta=(1-\varepsilon) n$, it is asymptotically better than the trivial bound as well.
\begin{cor}\label{cor:e-d}
For any graph $G$ on $n$ vertices with maximum degree $\Delta$
$$HG(G) \leq 1+\frac{3n}{4}+\frac{\Delta^2}{4n}.$$

In particular, for any $\varepsilon>0$ and any graph $G$ on $n$ vertices, for $n$ sufficiently large,  if $\Delta(G)\leq (1-\varepsilon )n,$ then $HG(G)\leq (1-\zeta)n$ for some $\zeta(\varepsilon)$ independent of $n$. 

\end{cor}
\subsection{Proof Outline}
In this subsection, we describe our plan of attack. Let $G$ be a graph where each vertex is given one of $k$ colors, and let $\cS$ be a winning guessing strategy for $G$. If a coloring $c$ is chosen at random, then by conditioning on the colors of its neighbors, it is easy to see that each vertex $v$ guesses their hat correctly with probability $1/k$. Letting $R_\cS(c)$ be the number of correct guesses on the coloring $c$ under the strategy $\cS$, by linearity of expectation, the expected number of vertices that guess correctly is:
\begin{obs}\label{obs:1.2Avg} 
$$\mathbb{E}[R_\cS(c)]=n/k$$
\end{obs} 
Therefore, if we bound $\mathbb{E}[R_\cS(c)]$ from below, we can obtain bounds on $k$.

Let $X_1$ be the set of all colorings $c \in \mathcal{C}$ where $R_\cS(c)=1$. That is, only one vertex in $G$ guesses their color correctly. Note that if $\cS$ is winning and $k$ is near $n$, nearly every coloring must be in $X_1$. Otherwise, the expectation $\mathbb{E}[R_\cS(c)]$ will exceed $n/k$. We will show that this is not the case: For any coloring $c$ in $X_1$, we will find many ``nearby'' colorings $c' \not \in X_1$ that only differ from $c$ in the color of one vertex. After showing that each $c'$ is not overcounted too many times, we will conclude that some constant proportion of the colorings are not in $X_1$, obtaining a lower bound on $\mathbb{E}[R_\cS(c)]$ and an upper bound on $k$. 


\section{Construction of the Auxiliary graph $H$}
Recall that $X_1$ is the set of all colorings of $G$ where exactly one vertex correctly guesses their hat color. To prove Theorem \ref{thm:main}, we use an auxiliary graph $H$ to understand which colorings not in $X_1$ are ``nearby" colorings in $X_1$. The vertex set of this graph is $\mathcal{C}$, the set of all possible $k$-colorings of $G$. It would be natural to decide the edges of $H$ by connecting colorings $c \in X_1$ and $c' \not \in X_1$ that differ in color on only one vertex of $G$. 

However, we want to ensure that the unique vertex $v$ that guessed correctly in $c$ also guesses correctly in $c'$, so we require that $c'$ does not change the color of any neighbor of $v$. We also require that the vertex that differs between $c$ and $c'$ guesses the new color correctly in $c'$. Formally, consider
the following auxiliary graph:

\begin{definition}
Let $G$ be a graph on $n$ vertices where each vertex is assigned one of $k$ colors, and let $\cS$ be a winning guessing strategy on $G$. Define the vertex set of the auxiliary graph $H$ as the set of all $k$-colorings of $G$, and call it $\mathcal{C}$. In $H$, colorings $c \in X_1$ and $c' \notin X_1$ are adjacent if for some $u \in V(G)$:
\begin{enumerate}[(P\arabic*)]
    \item If $v$ is the unique vertex that guessed correctly in $c$, $u \neq v$ and $uv \notin E(G)$. 
    \item For all $w \in V(G) \setminus \{u\}$, $c(w)=c'(w)$.
    \item $\cS$ guesses the color of $u$ correctly in $c'$. 
\end{enumerate}
\end{definition}

To properly analyze this auxiliary graph, we will need to more precisely describe the colorings not in $X_1$. For each $2 \leq i \leq n$, define:
$$X_i:= \left\{c \in \mathcal{C}:R_{\cS}(c)=i, \deg_H(c) \neq 0  \right\}$$ 
where $R_{\cS}(c)$ is the number of vertices that guess correctly when $G$ is colored by $c$. We exclude the colorings with degree $0$ in $H$ because they cannot be analyzed with our approach.
Let $Y$ be the set of these excluded colorings:
$$Y:=\mathcal{C} \setminus \bigcup_{i=1}^n X_i.$$
Note that the $X_i$ for ${1 \leq i \leq n}$, together with $Y$, partition the vertices of $H$, so: $$\sum_{i=1}^n |X_i|+|Y|=|\mathcal C|=k^n.$$
From these definitions we obtain the following equation:
 \begin{lemma}\label{2.2lem:trivialn/k} If $\frac{n}{k} < 2,$
        $$\frac{n}{k} \geq \frac{|X_1|+2|X_2|+\ldots +n|X_n|}{|X_1|+|X_2|+\ldots +|X_n|}$$
    \end{lemma}
    \begin{proof}
       Recall Observation \ref{obs:1.2Avg}, which states that $ \mathbb E[R_\cS (c)]=\frac{n}{k} $. Let $y=\mathbb E_{c\in Y}[R_\mathcal S(c)]$, and notice that $y\geq 2$, since each coloring not in $X_1$ guesses correctly on at least 2 vertices.\\  Therefore, 
        \begin{align*}
             \frac{n}{k} &=\mathbb E[R_\cS (c)] \\
             &=\frac{1}{|\mathcal C|} \sum_{c\in \mathcal C}R_\mathcal S(c) \\
            &=\frac{ \wXsum +y\,|Y|}{|X_1|+|X_2|+\ldots +|X_n|+|Y|}  \\
             &\geq \frac{|X_1|+2|X_2|+\ldots+ n|X_n|}{|X_1|+|X_2|+\ldots+ |X_n|}.
        \end{align*}
        Here the last inequality holds because $y\geq 2$ and we assumed $\frac{n}{k}< 2$. Note that if $Y$ is empty, the third and fourth lines are equivalent since $|Y|=0$.
    \end{proof}

\section{Analysis of $H$}
In this section, we will lower bound $\frac{n}{k}$ by double counting $e(H)$ and using Lemma \ref{2.2lem:trivialn/k}. (See Lemma \ref{3.7lem:bound}) Note that $H$ is bipartite with parts $X_1, \cC \setminus X_1$, since we only allow edges between $c\in X_1$ and $c'\notin X_1$. Therefore:
\begin{obs}\label{3.1obs:degs} $$e(H)=\sum_{c\in X_1} \deg(c)=\sum_{c' \notin X_1} \deg(c').$$\end{obs}
Using this observation, we obtain the following bounds: 
\begin{lemma}\label{3.2lem} 
     $$(n-1-\Delta)|X_1| \leq e(H) \leq \sum_{i=2}^n i(k-1)|X_i|.$$
\end{lemma}
\begin{proof}
Let $c \in X_1$, $c' \not \in X_1$. We will prove the lower and upper bounds by using Observation \ref{3.1obs:degs} and bounding $\deg (c)$ and $\deg(c')$ respectively.
\begin{claim}\label{3.3clm} For any $c\in X_1$, 
    $$\deg (c)\geq n-1-\Delta.$$
\end{claim}
\begin{proof}
    Let $c \in X_1$, and let $v \in V(G)$ be the vertex that guesses its color correctly in $c$. Let $u \in V(G)\setminus \{v\}$ be a vertex not adjacent to $v$, and let $a_u$ be the color that $u$ guesses under the coloring $c$. 
        Let $c_u \in \mathcal{C}$ be the coloring where:
        $$c_u(x):= \begin{cases} a_u \text{ if }x = u \\c(x) \text{ otherwise.}
        \end{cases}$$
        Then $c$ and $c_u$ are adjacent in $H$. For each $u \in V(G)$ not adjacent to $v$, the colorings $c_u$ are distinct, and there are at least $n-1-\Delta$ vertices not adjacent to $v$, not including $v$ itself. Thus, $\deg(c)\geq n-1-\Delta$. 
\end{proof}
\begin{claim}\label{3.4clm} For any $c'\in X_i$, $i\neq 1$,  
    $$\deg(c')\leq  i(k-1).$$
\end{claim}
\begin{proof}
    Let $c' \in X_i$, $i\neq 1$, and let $c \in X_1$ be adjacent to $c'$. Then there exist nonadjacent vertices $u$ and $v$ such that $v$ guesses correctly in $c$ (and $c'$), $u$ guesses correctly in $c'$, and $c'(w)=c(w)$ for all vertices $w \neq u$.
    In fact, an edge into $c'$ is determined by this unique vertex $u$ where the coloring changes, and that vertex must guess correctly in $c'$.

    For any fixed $c' \in X_i$, there are $i$ vertices that guess correctly in $c'$, so $i$ potential choices for $u$. For each of these vertices, there are $k-1$ ways to change its color. Since any $c$ adjacent to $c'$ must agree with $c'$ on every vertex except $u$, there are at most $i(k-1)$ colorings $c$ adjacent to $c'$. 
\end{proof}

Putting Claims \ref{3.3clm} and \ref{3.4clm} together with Observation \ref{3.1obs:degs}, we get 
$$e(H)=\sum_{c\in X_1}  \deg (c)\geq (n-1-\Delta)|X_1|,$$ and $$e(H)=\sum_{c' \notin X_1} \deg(c')\leq \sum _{c' \notin X_1, c' \in X_i} i(k-1)=\sum i(k-1)|X_i|.$$
This completes the proof of Lemma \ref{3.2lem}.
\end{proof}

Lemma \ref{3.2lem} allows us to compare $|X_1|$ and $2|X_2|+ 3|X_3|+ \ldots + n|X_n|$, but recall that to bound $\frac{n}{k}$ with Lemma \ref{2.2lem:trivialn/k}, we need to bound the fraction:
   $$ \frac{|X_1|+2|X_2|+ \ldots + n|X_n|}{|X_1|+|X_2| +\ldots + |X_n|}. $$ 

First, we bound the numerator, $|X_1|+2|X_2|+\ldots +n|X_n|$.

\begin{lemma}\label{3.5lem-num} 
$$\wXsum \geq\Xsum + \frac{1}{2}|X_1|\frac{n-1-\Delta}{k-1}$$
\end{lemma}

\begin{proof}
Dividing the inequality from Lemma \ref{3.2lem} by $k-1$, we obtain
\begin{equation}\label{eq:3.2dividedbyk}
     2|X_2|+3|X_3|+4|X_4|+\ldots +n|X_n| \geq \frac{|X_1|(n-1-\Delta) }{k-1}.
\end{equation}
Thus, 
\begin{align*}
    \wXsum &= (\Xsum) + |X_2|+2|X_3|+3|X_4|+\ldots +(n-1)|X_n| \\ 
    & = (\Xsum) +\frac{1}{2} \left(2|X_2|+4|X_3|+6|X_4|+\ldots +2(n-1)|X_n|\right) \\
    &\geq \Xsum + \frac{1}{2} \left( 2|X_2|+3|X_3|+4|X_4|+\;\ldots +n|X_n|\right)
    \\&\geq \Xsum +\frac{1}{2}|X_1|\frac{(n-1-\Delta)}{k-1}. &\text{by (\ref{eq:3.2dividedbyk})}
\end{align*}
\end{proof}

Now we bound the denominator $\Xsum$.
\begin{lemma}\label{3.6:lem-den} 
If $\frac n k < 2$, then
        $$\Xsum \leq \frac{|X_1|}{2-\frac{n}{k}}.$$
    \end{lemma}
    \begin{proof}
        Let $$a=\frac{|X_1|}{\Xsum},b=\frac{\tXsum}{\Xsum}.$$
       These give us the linear system: 
       \begin{align*}
           &a+b=\frac{\Xsum}{\Xsum}=1 \\
           &a+2b=\frac{|X_1|+2(|X_2|+|X_3|+\ldots+ |X_n|)}{|X_1|+|X_2|+|X_3|+\ldots+ |X_n|}\leq \frac{|X_1|+2|X_2|+3|X_3|+\ldots+ n|X_n|}{|X_1|+|X_2|+|X_3|+\ldots+ |X_n|}\leq \frac{n}{k}.
       \end{align*}
        The last inequality is from Lemma $\ref{2.2lem:trivialn/k}$, which we can apply since $\frac n k< 2$.
        Solving the linear system, we obtain that 
        $$a\geq 2-\frac{n}{k}.$$
        By definition, $a(\Xsum)=|X_1|$, and since $a\geq 2-\frac{n}{k}$, we have $$|X_1|=a(\Xsum)\geq (2-\frac{n}{k})(\Xsum), $$ so $$\Xsum \leq \frac{|X_1|}{2-\frac{n}{k}}.$$
    \end{proof}

Putting these last two lemmas together, we obtain:
\begin{lemma}\label{3.7lem:bound} 
    $$\frac n k \geq \min \left(2,\, 1+\left(2-\frac{n}{k}\right)\left(\frac{1}{2}\frac{n-1-\Delta}{k-1}\right)\right).$$
\end{lemma} 
\begin{proof}
If $\frac n k \geq 2,$ then we are trivially done. Thus, we may assume $\frac n k < 2$, allowing us the use of Lemma \ref{2.2lem:trivialn/k} and Lemma \ref{3.6:lem-den}.
    \begin{align*}
        \frac{n}{k} &\geq \frac{\wXsum}{\Xsum}  & \text{by \ref{2.2lem:trivialn/k}}
        \\ & = 1+\frac{|X_2|+2|X_3|+\ldots +(n-1)|X_n|}{\Xsum} 
        \\&\geq 1+ \left(2-\frac{n}{k}\right) \frac{|X_2|+2|X_3|+\ldots+ (n-1)|X_n|}{|X_1|} &\text{by \ref{3.6:lem-den}}
        \\& \geq 1+\left(2-\frac{n}{k}\right)\left(\frac{1}{2}\frac{n-1-\Delta}{k-1}\right). &\text{by \ref{3.5lem-num}}
    \end{align*}
\end{proof}
After lower bounding $\frac{n}{k}$, all that remains is to use basic algebraic manipulation to prove our main results.
\section{Proof of Main Results}
\begin{proof}[Proof of Theorem \ref{thm:main}]
    Let $G$ be a graph on $n$ vertices where each vertex is colored one of $k$ colors, and let $\cS$ be a winning strategy on $G$. 
    We will bound $k$ from above.
    
   We start by applying Lemma $\ref{3.7lem:bound}$. This gives $$\frac{n}{k}\geq \min \left(2, 1+\left(2-\frac{n}{k}\right)\left(\frac{1}{2}\frac{n-1-\Delta}{k-1}\right)\right).$$
   There are two possible cases:

   \medskip

   \noindent\textbf{If $2 \leq 1+\left(2-\frac{n}{k}\right)\left(\frac{1}{2}\frac{n-1-\Delta}{k-1}\right)$: }

    By Lemma \ref{3.7lem:bound}, $\frac n k \geq 2$. Observe that for all $n \geq 2$ and $0 \leq \Delta \leq n-1$, 
$$ k \leq \frac{n}{2} \leq  1+\frac{\Delta}{2}+ \frac{\sqrt{(\Delta +2)^2 + 2(n^2-3n-n\Delta)}}{2}.$$ 
    
   
   
    \noindent\textbf{If $ 2 \geq 1+\left(2-\frac{n}{k}\right)\left(\frac{1}{2}\frac{n-1-\Delta}{k-1}\right) $:} 
    
    In this case, Lemma \ref{3.7lem:bound} gives
    $$\frac{n}{k}\geq  1+\left(2-\frac{n}{k}\right)\left(\frac{1}{2}\frac{n-1-\Delta}{k-1}\right).$$
    Multiplying through by $2k(k-1)$, $$2n(k-1)\geq 2k(k-1)+(2k-n)\left(n-1-\Delta\right).$$ 

    Solving the quadratic for $k$, we get
    \begin{align*}
         k &\leq \frac{2\Delta+4+\sqrt{(2\Delta  +4)^2 +8(n^2-3n-n\Delta)}}{4}
         \\ &= 1+\frac{\Delta}{2}+ \frac{\sqrt{(\Delta +2)^2 + 2(n^2-3n-n\Delta)}}{2}.
    \end{align*}
    
\medskip

    In either case, it holds that $k \leq 1+\frac{\Delta}{2}+ \frac{\sqrt{(\Delta +2)^2 + 2(n^2-3n-n\Delta)}}{2}$. Because this equation holds for all winning strategies $\cS$, we conclude that $$HG(G) \leq 1+\frac{\Delta}{2}+ \frac{\sqrt{(\Delta +2)^2 + 2(n^2-3n-n\Delta)}}{2}.$$
\end{proof}

\begin{proof}[Proof of Corollary \ref{cor:Gn1,2}]
    It is well-known (see Bollobás, \cite{Bollobas2001}) that the maximum degree of the random graph is $(1+o(1))n/2$. Specifically, with high probability, $\Delta \leq \frac n 2 + (1+o(1)) \sqrt{\frac{n\log n}2}$. Let $\Delta = 
    \frac n 2+ \gamma$, where $\gamma=o(n)$. We apply Theorem \ref{thm:main} to obtain with high probability,
    \begin{align*}
        HG(G)&\leq 1+\frac{\Delta}{2}+ \frac{\sqrt{(\Delta +2)^2 + 2(n^2-3n-n\Delta)}}{2}
        \\&= 1+\frac \Delta 2 + \frac{\sqrt{\Delta ^2 +4\Delta +4+2n^2-2n\Delta -6n}}2
        \\&\leq 1+\frac n 4 + \frac \gamma 2+ \frac 1 2 \sqrt {\frac{n^2}{4}+n\gamma+\gamma^2+2n+4\gamma +4+2n^2-n^2-2n\gamma-6n}
        \\&\leq 1+\frac n 4 + o(n)+ \frac 1 2 \sqrt{\frac 5 4 n^2+o(n^2)}
        \\&\leq (1+o(1)) \frac{1+\sqrt 5}{4}n.
    \end{align*}
\end{proof}
\begin{proof}[Proof of Corollary \ref{cor:e-d}]
    Let $G$ be a graph on $n$ vertices. By \Cref{thm:main}, 
        \begin{align*}
        HG(G)\leq 1+\frac{\Delta}{2}+\frac{\sqrt{(\Delta+2)^2+2(n^2-3n-n\Delta)}}{2}.
        \end{align*}
        We note that $(\Delta+2)^2+2(n^2-3n-n\Delta)\leq n^2+(n-\Delta)^2$ since $\Delta\leq n-1$, so the expression becomes 
        $$HG(G)\leq 1+\frac \Delta 2+\frac{n}{2}\sqrt{1+\frac{(n-\Delta)^2}{n^2}}.$$ 
        Using the fact that $\sqrt{1+x^2}\leq 1+x^2/2$ for any $x \in \mathbb{R}$,
        \begin{align*} 
        HG(G)&\leq 1+\frac{\Delta}{2}+\frac{n}{2}\left(1+\frac{(n-\Delta)^2}{2n^2}\right)\\
        &=1+\frac{\Delta}{2}+\frac{n}{2}+\frac{n^2-2n\Delta+\Delta^2}{4n}=1+\frac{3n}{4}+\frac{\Delta^2}{4n}.
        \end{align*}
       To prove the last sentence of Corollary \ref{cor:e-d}, observe that if the maximum degree $\Delta$ of $G$ is less than $(1-\varepsilon)n$ for some $\varepsilon \in (0,1]$, $$HG(G)\leq 1+\frac{3n}{4}+\frac{(1-\varepsilon)^2n}{4}=1+\left(1-\frac \varepsilon 2+\frac{\varepsilon^2} 4\right)n.$$ 
       If we choose $\zeta(\varepsilon)= \varepsilon/8<\varepsilon/2-\varepsilon^2/4$, then for graphs $G$ on $n>8$ vertices, $HG(G) \leq (1-\zeta)n.$
\end{proof}

\bibliographystyle{plainnat}
\bibliography{refs}

\end{document}